\documentclass[preprint]{elsarticle}
\journal{Linear Algebra and its Applications}

\usepackage{amsmath, amsthm, amssymb}					
\usepackage[colorlinks=true]{hyperref}
\usepackage{mathrsfs}
\usepackage{physics}
\usepackage{enumitem}

\newtheorem{theorem}{Theorem}[section]
\newtheorem{lemma}[theorem]{Lemma}
\newtheorem*{lemma*}{Lemma}
\newtheorem{conj}[theorem]{Conjecture}
\newtheorem{corollary}[theorem]{Corollary}
\newtheorem{prop}[theorem]{Proposition}

\theoremstyle{definition}
\newtheorem{definition}[theorem]{Definition}
\newtheorem{example}[theorem]{Example}

\theoremstyle{remark}
\newtheorem{remark}[theorem]{Remark}

\newcommand{\diag}[1]{\operatorname{diag}\left( #1 \right)}		
\newcommand{\bb}[1]{\mathbb{#1}}						
\newcommand{\hyp}[2]{#1 \hyperref[#2]{\ref{#2}}}			

\begin{document}
\begin{frontmatter}
\title{On the realizability of the critical points of a realizable list\tnoteref{t1}}
\tnotetext[t1]{This work was supported by NSF Award \href{https://www.nsf.gov/awardsearch/showAward?AWD\_ID=1460699}{DMS-1460699}.}

\author[hc]{Sarah L.~Hoover\fnref{nsf}}
\ead{shoover@hamilton.edu}

\author[fit]{Daniel A.~McCormick\fnref{nsf}}										
\ead{dmccormick2014@my.fit.edu}

\author[uwb]{Pietro Paparella\corref{corpp}\fnref{nsf}}
\ead{pietrop@uw.edu}
\ead[url]{http://faculty.washington.edu/pietrop/}

\author[uwb]{Amber R.~Thrall\fnref{nsf}}
\ead{amber.rose.thrall@gmail.com}
\ead[url]{http://amber.thrall.me/}

\cortext[corpp]{Corresponding author.}

\address[hc]{Hamilton College, 198 College Hill Rd., Clinton, NY, 13323, USA}
\address[fit]{Florida Institute of Technology, 150 W. University Blvd., Melbourne, FL, 32901, USA}
\address[uwb]{Division of Engineering and Mathematics, University of Washington Bothell, 18115 Campus Way NE, Bothell, WA 98011, USA}

\begin{abstract}
The {nonnegative inverse eigenvalue problem} (NIEP) is to characterize the spectra of entrywise nonnegative matrices. A finite multiset of complex numbers is called {realizable} if it is the spectrum of an entrywise nonnegative matrix. Monov conjectured that the \(k\)\textsuperscript{th}-moments of the list of critical points of a realizable list are nonnegative. Johnson further conjectured that the list of critical points must be realizable. In this work, Johnson's conjecture, and consequently Monov's conjecture, is established for a variety of important cases including {Ciarlet spectra}, {Sule\u{\i}manova spectra}, spectra realizable via companion matrices, and spectra realizable via similarity by a complex Hadamard matrix. Additionally we prove a result on differentiators and trace vectors, and use it to provide an alternate proof of a result due to Malamud and a generalization of a result due to Kushel and Tyaglov on circulant matrices. Implications for further research are discussed.
\end{abstract}

\begin{keyword}
Nonnegative inverse eigenvalue problem \sep Sule\u{\i}manova spectrum \sep differentiator \sep trace vector \sep complex Hadamard matrix 

\MSC[2010] 15A29 \sep 15A18 \sep 15B48 \sep 30C15
\end{keyword}
\end{frontmatter}

\section{Introduction}

The longstanding \emph{nonnegative inverse eigenvalue problem} (NIEP) is to characterize the spectra of entrywise nonnegative matrices.  More specifically, given a finite multi-set (herein \emph{list}) $\Lambda = \{ \lambda_1, \dots, \lambda_n \}$ of complex numbers, the NIEP asks for necessary and sufficient conditions such that $\Lambda$ is the spectrum of an $n$-by-$n$ entrywise-nonnegative matrix. If \( \Lambda\) is the spectrum of a nonnegative matrix \(A\), then \(\Lambda\) is called \emph{realizable} and \(A\) is a called \emph{realizing matrix (for \( \Lambda \))}.

It is well-known that if \( \Lambda=\{\lambda_1,\dots,\lambda_n\} \) is realizable, then \( \Lambda \) must satisfy the following conditions:
\begin{align}
	\overline\Lambda &:= \{ \bar{\lambda}_1,\dots,\bar{\lambda}_n\} = \Lambda;			\label{selfconjugacy}	\\  	
  	\rho(\Lambda) &:= \max_{1 \le i\le n} \{|\lambda_i| \} \in \Lambda;  				\label{spectralradius}	\\		
    	s_k(\Lambda) &:= \sum_{i=1}^n \lambda_i^k \ge 0,~ \forall k \in \mathbb{N};~\text{and} 	\label{moment} 		\\  	
    	s_k^m(\Lambda) &\le n^{m-1}s_{km}(\Lambda),~\forall k,m \in \mathbb{N}.  			\label{J-LL}
\end{align}
It is worth noting that these conditions are not independent: Loewy and London \cite{loewylondon1978} showed that the \emph{moment condition} \eqref{moment} implies the \emph{self-conjugacy condition} \eqref{selfconjugacy}; Friedland \cite[Theorem 1]{friedland1978} showed that the \emph{eventual nonnegativity} of the moments implies the spectral radius condition \eqref{spectralradius}; and if the \emph{trace} is nonnegative, i.e., if \(s_1(\Lambda) \ge 0\), then the \emph{J-LL condition} \eqref{J-LL} (established independently by Johnson \cite{j1981} and by Loewy and London \cite{loewylondon1978}) implies the moment condition since 
\[ s_k(\Lambda) \ge \frac{1}{n^{k-1}} s_1^k(\Lambda) \ge 0. \]
Thus, the J-LL condition and the nonnegativity of the trace imply the self-conjugacy, spectral radius, and moment conditions. There is another necessary condition due to Holtz \cite{holtz2004}, which was shown to be independent of J-LL, but for brevity, it will not be considered here (for more information regarding the NIEP, there are several articles \cite{borobia1995,chu1998,egleston2004,niepSurvey,minc1988, soto2012} and monographs \cite{minc1988, bp1994} that discuss the problem). 

Let \( \Lambda = \{ \lambda_1, \dots, \lambda_n \} \subset \mathbb{C}\) (\(n \ge 2\)) be a list, let \( p (z) = \prod_{i=1}^n (z - \lambda_i) \), and let \( \Lambda' = \{ \mu_1,\dots, \mu_{n-1} \} \) denote the zeros of \( p' \). Monov \cite{monov2008} posed the following conjecture.

\begin{conj}[Monov] \label{conj:monov}
If \( \Lambda \) is realizable, then \( s_k (\Lambda') \ge 0 \) for all \( k \in \mathbb{N}\).
\end{conj}

The following is due to Johnson \cite{johnsonprivate}.

\begin{conj}[Johnson] \label{conj:johnson}
If \( \Lambda \) is realizable, then \( \Lambda' \) is realizable.
\end{conj}

It is clear that Johnson's conjecture implies Monov's conjecture. Cronin and Laffey \cite{croninlaffey2014} established Johnson's conjecture in the cases when \(n \le 4\), or when \(n \le 6\) and \( s_1(\Lambda)=0\), . 

Note that the converse of Conjecture \hyperref[conj:johnson]{\ref*{conj:johnson}} is not true; consider the list \( \Lambda = \{1, 1, -2/3,-2/3,-2/3 \}\)
and the polynomial \( p(z) = (z-1)^2(z+2/3)^3\). Then \( p'(z) = 5z^4 - 5z^2 - 20/27z +20/27 \), \( \Lambda' = \{ 1, 1/3, -2/3, -2/3 \} \), and \( \Lambda' \) is realizable \cite{loewylondon1978}. However, it is well-known that \( \Lambda \) is not realizable (see, e.g., Friedland \cite{friedland1978} and Paparella and Taylor \cite{pt2017} for a more general result). Note that there are instances when the constant of integration can be selected so that the zeros of an antiderivative form a realizable set (see Section \hyperref[sec:concrem]{\ref*{sec:concrem}}). 

In this work, we investigate and constructively prove Conjecture \hyperref[conj:johnson]{\ref*{conj:johnson}} for a variety of cases. First, we examine the \(k\)\textsuperscript{th} moments of the critical points of a realizable list and establish a sufficient condition for realizability via companion matrices. Next, we explore the d-companion matrix as a construction for a realizing matrix of the critical points and establish a sufficient condition that includes Ciarlet spectra \cite{ciarlet1968}. We then study differentiators following the work of Pereira \cite{pereira2003} and show that Johnson's conjecture holds for spectra realizable via complex Hadamard similarities. Along the way, we provide an alternate proof to a result of Malamud \cite{malamud2005} and extend the results on circulant matrices by Kushel and Tyaglov \cite{kusheltyaglov2016}. Additionally, we use these results to provide a new proof of a classical theorem on the interlacing roots and critical points for polynomials. Finally,  implications for further research are discussed.

\section{Notation}

Unless otherwise stated, \( \Lambda := \{ \lambda_1, \dots, \lambda_n \} \subset \mathbb{C}\) (\(n \ge 2\)), \( p (t) := \prod_{i=1}^n (t - \lambda_i) \in \bb{C}[t]\), and \( \Lambda' := \{ \mu \in \mathbb{C} \mid p'(\mu) = 0 \} = \{ \mu_1,\dots, \mu_{n-1} \} \). Furthermore, we refer to the \emph{critical points} of \(p\), i.e., the zeros of \( p' \),  as the critical points of \( \Lambda \). 

For \( n \in \mathbb{N}\), we let \(\langle n \rangle\) denote the set \(\{1, \dots, n\}\) .  

We let \(M_{m\times n}(\mathbb{F})\) denote the set of \(m\text{-by-}n\) matrices with entries from a field \(\mathbb{F}\). In the case when \( m = n\), \(M_{m\times n}(\mathbb{F})\) is abbreviated to \(M_n(\mathbb{F})\).

For \(A\in M_n(\mathbb{F}) \), we let \( \sigma(A) \) denote the \emph{spectrum} (i.e., list of eigenvalues) of \(A\);  we let \(A_{(i)}\) denote the \(i\)\textsuperscript{th} \emph{principal submatrix} of \(A\), i.e.,  \(A_{(i)}\) is the \((n-1)\)-by-\((n-1)\) matrix obtained by deleting the \(i\)\textsuperscript{th}-row and \(i\)\textsuperscript{th}-column of \(A\); and we let \(\tau(A)\) denote the \emph{normalized trace of \(A\)}, i.e., \( \tau(A):= (1/n) \trace A\). 

We let \(\diag{\lambda_1,\dots,\lambda_n}\) denote the \emph{diagonal matrix} whose \(i\)\textsuperscript{th} diagonal entry is \( \lambda_i\).

For \(A \in M_m(\mathbb{F})\) and \(B \in M_n(\mathbb{F})\), the \emph{direct sum of \(A\) and \(B\)}, denoted \(A \oplus B\), is the matrix
\[
	\begin{bmatrix}
 		A & 0 \\
 		0 & B
 	\end{bmatrix} \in M_{m\times n}(\mathbb{F}).
\]
The \emph{Kronecker product} of \( A\in M_{m\times n}(\mathbb{F})\) and \( B \in M_{p\times q}(\mathbb{F})\), denoted \( A \otimes B \), is the \( (mp) \text{-by-} (nq) \) block matrix
\[
  \begin{bmatrix} a_{11}B & \dots & a_{1n}B \\ \vdots & \ddots &
    \vdots \\ a_{m1}B & \dots & a_{mn}B
  \end{bmatrix}.
\] 

\section{Preliminary Results}\label{sec:prelim}

\begin{prop}
	\label{thm:selfconjugate}
		If \(\Lambda\)  is realizable, then \( \Lambda' = \overline{\Lambda'} \).
\end{prop}

\begin{proof}
	Follows from the fact that the derivative of a real polynomial  is real.
\end{proof}

\begin{prop}
	\label{thm:1stmoment}
		If \(\Lambda\)  is realizable, then \(s_1(\Lambda') \ge 0\).
\end{prop}

\begin{proof}
Since
	\[\frac{1}{n} s_1(\Lambda)= \frac{1}{n-1} s_1(\Lambda'), \]
it follows that 
	\[ s_1(\Lambda') = \frac{n-1}{n} s_1(\Lambda) \ge 0. \qedhere \] 
\end{proof}

Monov \cite[Proposition 2.3]{monov2008} provides a formula that relates the moments of a list with the moments of its critical points. 

\begin{theorem}[Monov's moment formula]
\label{thm:monov}
If
\[
    	d_k (\Lambda) :=
    	\begin{vmatrix}
      	s_1(\Lambda)          & n               & 0                 & \dots & 0      \\
      	2s_2(\Lambda)         & s_1(\Lambda)    & n                 & \dots & 0      \\
      	\vdots                & \vdots          & \vdots            &       & \vdots \\
      	(k-1)s_{k-1}(\Lambda)  & s_{k-2}(\Lambda)  & s_{k-3}(\Lambda)  & \dots & n       \\
      	ks_k(\Lambda)         & s_{k-1}(\Lambda)  & s_{k-2}(\Lambda)  & \dots & s_1(\Lambda)
    	\end{vmatrix},
  \]
then
\begin{equation}
s_k(\Lambda') = s_k(\Lambda) +\left(-\frac{1}{n}\right)^k d_k (\Lambda). \label{monov}
\end{equation}
\end{theorem}

\begin{corollary}
If \(\Lambda\) is realizable, then \(s_2(\Lambda') \ge 0\).  
\end{corollary}

\begin{proof}
A straightforward application of Theorem \ref{thm:monov} yields
\[
	s_2(\Lambda') = \frac{n-2}{n}s_2(\Lambda) + \frac{1}{n^2}s_1^2(\Lambda) \ge 0. \qedhere 
\]
\end{proof}

With this we see a particular case of the J-LL condition satisfied.

\begin{corollary}
  	\label{cor:JLL12}
  		If \(\Lambda\) is a list, then 
  		\begin{equation} 
  		s_1^2(\Lambda) \le n s_2(\Lambda) 			\label{jll21}
  		\end{equation}
  		if and only if 
  		\[ s_1^2(\Lambda') \le  (n - 1)s_{2}(\Lambda'). \]
\end{corollary}

\begin{proof}
Following \eqref{monov}, 
\begin{align*}
	s_1^2(\Lambda) \le n s_2(\Lambda) 																			
	&\Longleftrightarrow 0 \le \frac{n-2}{n} s_2(\Lambda) - \frac{n-2}{n^2}s_1^2(\Lambda) 											\\
	&\Longleftrightarrow 0 \le \frac{n-2}{n} s_2(\Lambda) + \frac{1}{n^2}s_1^2(\Lambda) - \frac{1}{n^2}s_1^2(\Lambda) - \frac{n-2}{n^2}s_1^2(\Lambda) 	\\
	&\Longleftrightarrow 0 \le \frac{n-2}{n} s_2(\Lambda) + \frac{1}{n^2}s_1^2(\Lambda) - \frac{n-1}{n^2} s_1^2 (\Lambda)						\\
	&\Longleftrightarrow 0 \le (n-1) \left(\frac{n-2}{n} s_2(\Lambda) + \frac{1}{n^2}s_1^2(\Lambda)\right) - \left(\frac{n-1}{n} s_1(\Lambda)\right)^2 		\\
	&\Longleftrightarrow 0\le (n-1) s_2(\Lambda') - s_1^2(\Lambda')															\\
	&\Longleftrightarrow s_1^2(\Lambda') \le (n-1) s_2(\Lambda'). \qedhere
\end{align*}
\end{proof}

Inequality \eqref{jll21} is important for the realizability of lists with three elements \cite[Theorem 2]{loewylondon1978}, as well as \emph{generalized Sule\u{\i}manova spectra} \cite{laffeysmigoc2006}, which we define in the sequel.

A real list \(\Lambda\) is called a \emph{Sule\u{\i}manova spectrum} if \( s_1(\Lambda)\ge0 \) and \(\Lambda\) contains exactly one nonnegative entry.  A list \(\Lambda\) is called a \emph{generalized Sule\u{\i}manova spectrum} when \( s_1(\Lambda)\ge0 \) and \(\Lambda\) contains exactly one entry with nonnegative real part. 

Friedland \cite{friedland1978} and Perfect \cite{perfect1953} proved that all Sule{\u\i}manova spectra are realizable by a companion matrix. In particular, they showed that if $\Lambda$ is a Sule{\u\i}manova spectrum and 
\begin{equation*} 
p(t) = \prod_{i=1}^n (t - \lambda_i) = t^n+\sum_{k=1}^n a_{n-k} t^{n-k},
\end{equation*} 
then \( a_k \le 0\) for every \( k \in \langle n \rangle\). Consequently, the companion matrix
\begin{equation*}
C(p) := 
\begin{bmatrix}
0 & -a_0 \\
I_{n-1} & -a
\end{bmatrix},~a := \begin{bmatrix} a_1 & \cdots & a_{n-1} \end{bmatrix}^\top
\end{equation*} 
is nonnegative. 

Note that the class of spectra realizable via companion matrices contains spectra other than (generalized) Sule{\u\i}manova spectra (e.g., the \(n\)\textsuperscript{th} roots of unity for \(n \ge 4\)).

\begin{theorem}
\label{thm:sulecrit}
If \(\Lambda\) is a list such that \(C(p)\) is nonnegative, then \(\Lambda'\) is realizable. Furthermore, \(\Lambda'\) is realizable by a companion matrix.
\end{theorem}

\begin{proof}
By hypothesis, 
\begin{equation*} 
p(t) = t^n+\sum_{k=1}^n a_{n-k} t^{n-k},
\end{equation*} 
with \( a_k \le 0\) for every \( k \in \langle n \rangle\). Thus, the non-monic coefficients of \( \frac{1}{n}p'\) are nonpositive, i.e., \(C(p'/n) \ge 0\).
\end{proof}

\begin{remark}
If \(p\) is a polynomial such that \(C(p)\) is nonnegative, then, for \( 1 \le k \le n-1\), the zeros of \(p^{(k)}\) are realizable. Moreover, if the constants of integration are chosen to be nonpositve, then the roots of its successive anti-derivatives are also realizable (see Section \hyperref[sec:concrem]{\ref*{sec:concrem}}).
\end{remark}

\begin{theorem}
If \(\Lambda\) is a realizable generalized Sule\u{\i}manova spectrum such that \(s_1(\Lambda) = 0\), then \(\Lambda'\) is realizable.
\end{theorem}

\begin{proof}
Laffey and \v{S}migoc \cite{laffeysmigoc2006} showed that \(\Lambda\) is realizable by a matrix of the form \(C + \alpha I\), where \(C\) is a trace-zero companion matrix and \(\alpha\) is a nonnegative scalar. Since \(0 = s_1(\Lambda) = \trace (C + \alpha I) = \trace C + \alpha \trace I = n\alpha\), it follows that \(\alpha = 0\). Thus, \(\Lambda\) is realizable by a trace-zero companion matrix, so \(\Lambda'\) is realizable.
\end{proof}

The following theorem is a classical result on the relationship between  polynomial and its critical points when its zeros are real. 

\begin{theorem}
\label{thm:interlace}
Let \(p\) be a polynomial with real roots \(\lambda_1,\dots,\lambda_n\) and critical points \(\mu_1,\dots,\mu_{n-1}\), each listed in descending order. Then the critical points of \(p\) interlace the roots of \(p\), i.e.,
	\[
		\lambda_1 \ge \mu_1 \ge \lambda_2 \ge \dots \lambda_{n-1} \ge \mu_{n-1} \ge \lambda_n.
	\]
\end{theorem}

Theorem \ref{thm:interlace} can be proved via Rolle's theorem and some simple root counting arguments, but the details are somewhat cumbersome. Instead, we defer this proof to Stection \ref{sec:diff} to highlight the utility of the results presented there.

While Theorem \ref{thm:sulecrit} proves that the critical points of Sule\u{\i}manova spectra are realizable, the following is worth noting.

\begin{theorem}
If \(\Lambda\) is a Sule{\u\i}manova spectrum, then \(\Lambda'\) is a Sule\u{\i}manova spectrum.
\end{theorem}

\begin{proof}
Since \(p\) has all real roots, by Theorem \ref{thm:interlace} the critical points must all be real and interlace the zeros. Thus, there is at most one nonnegative critical point. Furthermore, by Theorem \ref{thm:1stmoment}, \(s_1(\Lambda') \ge 0\), so there must be at least one nonnegative critical point.
\end{proof}

\section{The D-Companion Matrix}

Just as the roots of a polynomial can be studied with a companion matrix, Cheung and Ng \cite{cheungng2006} introduced the d-companion matrix to study the critical points of a polynomial.

\begin{definition}[d-companion matrix]
Let \(p(t) = \prod_{i=1}^n (t - \lambda_i)\) be a polynomial of degree \(n\ge 2\) and let \( D := \diag{\lambda_2,\dots,\lambda_n} \in M_{n-1}(\mathbb{C})\). If \(I\) and \(J\) denote the \((n-1)\text{-by-} (n-1)\) identity and all-ones  matrices respectively, then the the matrix
\[ A = D \left(I - \frac{1}{n} J\right) + \frac{\lambda_1}{n} J \]
is called a \emph{d-companion matrix of $p$}. 
\end{definition}

Notice that, after relabeling the zeros of $p$, there are $n$ possible choices for the d-companion matrix. Cheung and Ng \cite[Theorem 1.1]{cheungng2006} showed that the spectrum of any d-companion matrix of a polynomial coincides with its multiset of critical points.

A straightforward calculation reveals that
\begin{equation}
\label{eq:DCompEnt}
a_{ij} = \frac{1}{n}
\begin{cases}
 (n-1) \lambda_{i+1} + \lambda_1 & i = j \\
 \lambda_1 - \lambda_{i+1} & i \ne j.
\end{cases}
\end{equation}

\begin{remark}
If \(\Lambda \subset \mathbb{R}\) is realizable, with \( \lambda_1 = \rho(\Lambda)\), then \eqref{eq:DCompEnt} reveals that \(\Lambda'\) is the spectrum of a \emph{Metzler matrix}, i.e., \(\Lambda'\) is realizable by a matrix of the form \(A=B-\alpha I\), with \(B \ge 0\) and \(\alpha \ge 0\).
\end{remark}

Ciarlet \cite{ciarlet1968} showed that the following condition is sufficient for realizability.

\begin{theorem}[Ciarlet]
  Let \(\Lambda = \{\lambda_1,\dots,\lambda_n\} \subset \mathbb{R}\)
  such that \(\lambda_1 \ge \dots \ge \lambda_n\).
  If \(n\lambda_n + \lambda_1 \ge 0\), then \(\Lambda\) is realizable.
\end{theorem}

From the d-companion matrix, we derive a similar necessary condition
for the realizability of critical points.

\begin{theorem}
 \label{thm:DCompIneq}
 Let \(\Lambda = \{\lambda_1,\dots,\lambda_n\} \subset \mathbb{R}\) be realizable and suppose that \(\lambda_1 \ge \dots \ge \lambda_n\).  If \((n-1)\lambda_n + \lambda_1 \ge 0\), then \(\Lambda'\) is realizable.
\end{theorem}

\begin{proof}
Immediate in view of \eqref{eq:DCompEnt}.
\end{proof}

\begin{corollary}
  	The critical points of a Ciarlet spectrum are realizable.
\end{corollary}

Although not directly applicable to the primary problem considered here, given a polynomial with complex roots, a construction of the d-companion matrix with real entries may be useful elsewhere.

\begin{theorem}
If \(\Lambda = \{\lambda_1,\dots,\lambda_m,\mu_1, \bar\mu_1,\dots,\mu_n,\bar\mu_n \}\), where \(\Im(\lambda_i) =0,~\forall i \in \langle m \rangle \) and \( \Im(\mu_i)\ne 0,~\forall i\in \langle n \rangle\), then the spectrum of the matrix (whose entries are real)
\[ B \left(I - \frac{1}{m+2n} K\right) + \frac{\lambda_1}{m + 2n} K, \]
is the critical points of \(\Lambda\), where

  \[
    B =
    \begin{bmatrix}
      \lambda_2           \\
      & \ddots            \\
      & &  \lambda_{m}
    \end{bmatrix}
    \oplus
    \left(
      \bigoplus_{i=1}^n
      \begin{bmatrix}
        \Re\mu_i & \Im\mu_i \\
        -\Im\mu_i & \Re\mu_i
      \end{bmatrix}
    \right),
  \]

  \[
    K =
    \begin{bmatrix}
      J_{(m-1)\times(m-1)} & G \\
      G^\top & C
    \end{bmatrix},
  \]

  \[
    G = J_{(m-1) \times n} \otimes
    \begin{bmatrix}
      \sqrt{2} & 0
    \end{bmatrix},
  \]

  \noindent
  and

  \[
    C = J_{n\times n} \otimes
    \begin{bmatrix}
      2 & 0 \\
      0 & 0
    \end{bmatrix}.
  \]
\end{theorem}

\begin{proof}
If
  \[
    V := \frac{1}{\sqrt{2}}
    \begin{bmatrix}
      1 & 1 \\
      i & -i
    \end{bmatrix},
  \]
then \(V\) is unitary and, for any complex number \(z\),
  \[
    V \diag{z,\bar{z}}V^* =
    \begin{bmatrix}
      \Re z & \Im z \\
      -\Im z & \Re z
    \end{bmatrix}.\]

If \(A\) is the the d-companion matrix \(U := I_{m\times m} \oplus \bigoplus_{i=1}^n V\), then \(U\) is unitary and \(UAU^*\) is the desired matrix.
\end{proof}

\section{Differentiators and Complex Hadamard Matrices}\label{sec:diff}

Pereira \cite{pereira2003}, following the work of Davis \cite{davis1959}, studied differentiators and introduced the concept of a trace vector to solve several unsolved conjectures in the geometry of polynomials. Although Pereira studied differentiators and trace vectors in the setting of linear operators on finite dimensional Hilbert spaces, we only consider these objects for matrices and summarize some of his results in this setting.

\begin{definition}
For \(A \in M_n(\mathbb{C})\) and a unit vector \(z \in \mathbb{C}^n\), let \(P = P(z) := I - zz^* \in M_n(\mathbb{C})\).  If \(B := PAP|_{P\mathbb{C}^n}\) (the matrix \(B\) is called the \emph{compression of \(A\) onto \(P\mathbb{C}^n\)}), then \(P\) is called a \emph{differentiator (of \(A\))} if
\[ p_B(t) = \frac{1}{n} p'_A(t), \] 
where \(p_M\) denotes the the characteristic polynomial of \(M \in M_n(\mathbb{C})\).
\end{definition}

\begin{definition}
Let \(A \in M_n(\mathbb{C})\) and \(z\in \mathbb{C}^n\). Then \(z\) is called a \emph{trace vector (of \(A\))} if \(z^*A^kz = \tau(A^k)\), for every nonnegative integer \(k\).
\end{definition}

\begin{remark}
Corresponding to the case when \(k=0\), it is clear that all trace vectors have unit length. 
\end{remark}

The following result is due to Pereira.

\begin{theorem}[Pereira {\cite[Theorem 2.5]{pereira2003}}]
If \(A \in M_n(\mathbb{C})\) and \(z \in \mathbb{C}^n\) is a unit vector, then \(P\) is a differentiator of \(A\) if and only if \(z\) is a trace vector of \(A\).
\end{theorem}

In addition, Pereira \cite[Theorem 2.10]{pereira2003} shows that every matrix possesses a trace vector, i.e., every matrix possesses a differentiator. 

\begin{example}
If \(e_i\) is a trace vector of \(A \in M_n(\mathbb{C})\), where \(e_i\) (\(1 \le i \le n\)) denotes the \(i\textsuperscript{th}\) canonical basis vector of \(\mathbb{C}^n\), and \(P = I - e_ie_i^\top\), then the compression of \(A\) onto \(P\mathbb{C}^n\) is \(A_{(i)}\).
\end{example}

\begin{prop}
 Let \(A \in M_n(\mathbb{C})\) and let \(z\) be a trace vector of \(A\). If \(\alpha\) is a complex number such that \(|\alpha| = 1\), then \(\alpha z\) is a trace vector of \(A\).
\end{prop}

\begin{proof}
For any nonnegative integer \(k\), notice that
\[
(\alpha z)^*A^k (\alpha z) = \bar{\alpha} \alpha z^*A^kz = z^*A^kz = \tau(A^k). \qedhere
\]
\end{proof}

A stronger result is available for diagonal matrices. 

\begin{lemma}
\label{lem:diagtrace}
Let \(D \in M_n(\mathbb{C})\) be a diagonal matrix and \(z \in \mathbb{C}^n\). If \(|z_i|=\frac{1}{\sqrt n}\) for all \(i \in \langle n \rangle\), then \(z\) is a trace vector of \(D\). 
\end{lemma}

\begin{proof}
If \(k\) is a nonnegative integer, then
\[
z^* D^k z = \sum_{i=1}^n \overline{z}_i  z_i  d_{ii}^k = \sum_{i=1}^n |z_i|^2  d_{ii}^k = \frac{1}{n} \sum_{i=1}^n d_{ii}^k = \tau(D^k), 
\]
i.e., \(z\) is a trace vector of \(D\).
\end{proof}

\begin{remark}
The converse of Lemma \ref{lem:diagtrace} fails; indeed, if \(D\) is a diagonal matrix, then 
\[ \left(\frac{e_i}{\sqrt{n}}\right)^* D \left(\frac{e_i}{\sqrt{n}}\right) = \frac{d_{ii}}{n} = \tau(D),~\forall i \in \langle n \rangle. \]
\end{remark}

Although not relevant to this work, we note the following result. 

\begin{prop}
If \(z \in \mathbb{C}^n\) is a trace vector of every \(n\)-by-\(n\) diagonal matrix, then \(|z_i|=\frac{1}{\sqrt n}\) for all \(i \in \langle n \rangle\).
\end{prop}

\begin{proof}
Suppose that \(z \in \mathbb{C}^n\) is a trace vector of all diagonal matrices and let \( i \in \langle n \rangle\). If \(D = e_i e_i^\top\), then 
\[ \frac{1}{n} = \tau(D^k) = z^* D^k z = \sum_{j=1}^n \overline{z}_j  z_j  d_{jj}^k = |z_i|^2 ,\]
and the result is established upon taking square-roots throughout.
\end{proof}

\begin{lemma}
\label{lem:scalsim}
If \(Q = \alpha S\), with \(\alpha \in\mathbb{C}\backslash\{0\}\) and  \(S \in {GL}_n(\mathbb{C})\), then for \(A \in M_n(\mathbb{C})\), \(QAQ^{-1} = SAS^{-1}\).
\end{lemma}

\begin{proof}
\( QAQ^{-1} = (\alpha S)A(\alpha S)^{-1} = \frac{\alpha}{\alpha}SAS^{-1} = SAS^{-1}\).
\end{proof}

Malamud \cite[Proposition 4.2]{malamud2005} provides the following result on the relationship between a polynomial and its critical points.

\begin{prop}[Malamud]
\label{prop:malamud}
If \(p\) is a polynomial, then there is a normal matrix \(A\) such that \(\sigma(A)=\Lambda\) and \(\sigma(A_{(n)}) = \Lambda'\).
\end{prop}

In particular, Malamud shows that if \(A = UDU^*\), where \(D = \diag{\lambda_1, \dots, \lambda_n}\) and \(U\) is a unitary matrix satisfying 
\begin{equation}
|u_{nj}|=\frac{1}{\sqrt{n}},~\forall j \in \langle n \rangle, \label{malprop}
\end{equation}
then the matrix \(A\) satisfies the conclusion of Proposition \ref{prop:malamud}.

Following the work of Pereira, we provide an alternate proof to the above result and generalize it to any principal submatrix of \(A\).

\begin{theorem}
\label{thm:genmal}
Let \(\Lambda = \{\lambda_1,\dots,\lambda_n\}\) be a list, and let \(U \in M_n(\mathbb{C})\) be a unitary matrix such that for some \(i \in \langle n \rangle\), \(|u_{ij}| = \frac{1}{\sqrt{n}}\) for all \(j \in \langle n \rangle\). If \(A = UDU^*\), where \(D = \diag{\lambda_1, \dots, \lambda_n}\), then \(\sigma \left(A_{(i)} \right) = \Lambda'\).
\end{theorem}

\begin{proof}
Suffices to show that \(e_i\) is a trace vector of \(A\). Notice that \(U^* e_i\) is the entrywise complex conjugate of the \(i\)\textsuperscript{th}-row of \(U\). By hypothesis, each element of this row has modulus \(\frac{1}{\sqrt{n}}\), thus, by \hyp{Lemma}{lem:diagtrace}, \(U^* e_i\) is a trace vector of \(D\). Moreover, if \(k\) is any nonnegative integer, then 
\begin{align*} 
e_i^* A^k e_i = e_i^*\left(UD^k U^*\right)e_i = (U^* e_i)^* D^k (U^* e_i) = \frac{1}{n}\trace D = \frac{1}{n}\trace A = \tau(A,    	
\end{align*}
i.e., \(e_i\) is trace vector for \(A\).
\end{proof}

A proof of Proposition \ref{prop:malamud} is immediate once a unitary matrix is provided that satisfies property \eqref{malprop}. Recall that an $n$-by-$n$ matrix \(H\) is called a \emph{Hadamard matrix (of order $n$)} if \(h_{ij} \in \{ \pm 1 \}\) and \(HH^\top = nI\). A matrix \(H\) is called a \emph{complex Hadamard matrix (of order $n$)} if \(|h_{ij}| = 1\) and \(HH^* = nI\). Notice that for any complex Hadamard matrix \(H \in M_n({\mathbb{C}})\), the matrix \(U := \frac{1}{\sqrt{n}}H\) is unitary. 

For a fixed positive integer \(n\) and the complex scalar \(\omega := \exp(-2\pi i/n)\), the matrix 
\[
 F = F_n :=
\begin{bmatrix}
1 & 1 & 1 & \dots & 1 \\
1 & \omega & \omega^2 & \dots & \omega^{n-1} \\
1 & \omega^2 & \omega^4 & \dots & \omega^{2(n-1)} \\
      \vdots & \vdots & \vdots & \ddots & \vdots \\
      1 & \omega^{n-1} & \omega^{2(n-1)} & \dots & \omega^{(n-1)^2}
    \end{bmatrix},
\]
is called the \emph{discrete Fourier transform (DFT) matrix (of order $n$)}. As is well-known, the DFT matrices are complex Hadamard matrices. Since \(|\omega^k| = |\omega|^k=1^k = 1\) for every integer \(k\), it follows that every row of the unitary matrix \(U:=\frac{1}{\sqrt{n}}F_n\) satisfies the condition required in Proposition \ref{prop:malamud}. Furthermore, if \(\Lambda = \{\lambda_1,\dots,\lambda_n\}\) is a list and \(A = U D U^*\), \(D = \diag{\lambda_1,\dots,\lambda_n}\), then, following Theorem \ref{thm:genmal}, \(\sigma(A_{(i)}) = \Lambda'\), for every \( i \in \langle n \rangle\).

For \( S \in GL_n(\bb{C})\), the set 
\[ \mathcal{C}(S) := \{ x \in \bb{C}^n \mid SDS^{-1} \ge 0,~D=\diag{x_1,\dots,x_n} \} \] 
is called the \emph{spectracone of S} \cite{jp_psniep,johnsonpaparella2016}. It is known that the spectracones corresponding to Hadamard matrices comprise a large class of realizable real spectra \cite{johnsonpaparella2016}. Furthermore, spectracones corresponding to DFT matrices comprise a large class of realizable nonreal spectra \cite{jp_psniep}.

\begin{corollary}
Let \(\Lambda = \{\lambda_1,\dots,\lambda_n\}\) be a list and \(H\) be a complex Hadamard matrix. If \(A = UDU^* \ge 0\), where \(U:=\frac{1}{\sqrt{n}}H\) and \(D := \diag{\lambda_1,\dots,\lambda_n}\), then \(\Lambda'\) is realizable. In particular, \(A_{(i)}\) is a realizing matrix for \(\Lambda'\).
\end{corollary}

\begin{remark}
There is further evidence that Conjectrure \ref{conj:johnson} is true; recall that Boyle and Handelman \cite{boylehandelman1991} found necessary and  sufficient conditions for \(\tilde{\Lambda}=\Lambda \cup \{0,\dots,0\}\) to be the spectrum of a nonnegative matrix. Laffey \cite{laffey2012} gave a construction for a realizing \(n\text{-by-}n\) matrix \(X\). Laffey et al. \cite{laffeyloewysmigoc2016} showed that \(e_n\) is a trace vector of \(X\), thus the critical points of \(\tilde{\Lambda}\) are realizable by \(X_{(n)}\).
\end{remark}

A matrix of the form 
\begin{align}
C =
\begin{bmatrix}
c_1 & c_2 & \cdots & c_n \\
c_n & c_1 & \cdots & c_{n-1} \\
\vdots & \vdots & & \vdots \\
c_2 & c_3 & \cdots & c_1 
\end{bmatrix}  \label{circ}
\end{align}
is called a \emph{circulant matrix} or \emph{circulant}. It is well-known that a matrix \(C\) is a circulant if and only if there is a diagonal matrix \(D\) such that \(C = U D U^*\), where \(U := \frac{1}{\sqrt{n}}F_n\) \cite[Theorems 3.2.2 and 3.2.3]{davis1979}. Kushel and Tyaglov \cite[Theorem 2.1]{kusheltyaglov2016} showed that if \(p\) is the characteristic polynomial of \(C\) and \(\Lambda = \sigma(C)\), then \(\sigma(C_{(n)}) = \Lambda'\). However, Theorem \ref{thm:genmal} yields a more general result with a simpler proof. 

\begin{corollary}
If \(C\) is defined as in \eqref{circ}, \(\Lambda = \sigma(C)\), then \(\sigma(C_{(i)}) = \Lambda'\), for every \(i \in \langle n \rangle\).
\end{corollary}

\begin{proof}
Immediate in view of Theorem \ref{thm:genmal} and Lemma \ref{lem:scalsim}.
\end{proof}

We conclude this section with a proof of Theorem \ref{thm:interlace}.

\begin{proof}[Proof of Theorem \ref{thm:interlace}]
Let \(\Lambda\) be a real list.  Set \(A = H\diag{\lambda_1, \dots, \lambda_n}H^{-1}\), where \(H\) is any complex Hadamard matrix. By Lemma \ref{lem:scalsim}, \(A\) is a unitary similarity of a real diagonal matrix, so it is Hermitian. Furthermore, by Theorem \ref{thm:genmal}, \(\sigma\left(A_{(i)}\right) = \Lambda'\), for every \( i \in \langle n \rangle \). By the Cauchy interlacing theorem \cite[Theorem 4.3.17]{hornjohnson2012}, \(\Lambda'\)  interlaces \(\Lambda\).
\end{proof}

\section{Concluding Remarks}\label{sec:concrem}

Since Johnson's conjecture is stronger than Monov's conjecture, we have shown both Johnson's and Monov's conjecture for Sule{\u\i}maonva spectra, Ciarlet spectra, spectra realizable by companion matrices, and spectra realizable by complex Hadamard Similarities. In the general case, Johnson's conjecture remains open for \(n \ge 5\).

Since the J-LL and trace conditions imply the moment condition, proving that J-LL is satisfied for the critical points of realizable lists is a worthwhile avenue for investigation, as this would prove Monov's conjecture.

Cheung and Ng \cite{cheungng2010}, recognize that other constructions of a d-companion matrix exist. They provide methods for constructing such matrices which have already proven fruitful in the study of circulant matrices \cite{kusheltyaglov2016}. Future research on finding alternative constructions to the d-companion matrix may be critical to solving this conjecture.

It is simple to show that, for spectra realizable via companion matrices, there exists an anti-derivative with roots that are realizable as well. Notice that if
\( p(t) = t^n + \sum_{k=1}^n a_{n-k} t^{n-k} \), where each \(a_k \le 0\),
then any anti-derivative is of the form
\[
P(t) = \frac{1}{n+1} t^{n+1} + \sum_{k=1}^n \frac{1}{n-k+1} a_{n-k} t^{n-k+1}+c, \quad a_k\le0,
\]
Notice that \((n+1)P\) is monic and its companion matrix is nonnegative if and only if \(c\) is nonpositive. Thus, there are infinitely-many anti-derivatives whose zeros form a realizable list. 

Bhat and Mukherjee \cite{bhatmukherjee2007} introduced integrators as an analog to differentiatos. While every matrix has a differentiator \cite{pereira2003}, the existence of an integrator is neither guaranteed nor unique when an integrator is available. An explicit construction of an integrator is an avenue worthy of further exploration.

\section*{Acknowledgements}

The authors would like to thank the National Science Foundation for funding the REU program and the University of Washington Bothell for hosting the research program. In particular, we would like to thank Professors Jennifer McLoud-Mann and Casey Mann for their efforts.

\bibliography{references}
\bibliographystyle{abbrv}

\end{document}